\documentclass{article}
\usepackage[utf8]{inputenc}
\usepackage{amsmath}
\usepackage{amsfonts}
\usepackage{thmtools}
\usepackage{amssymb}
\usepackage{mathtools}
\usepackage{amsthm}
\usepackage{xcolor}
\usepackage{tikz-cd}

\usepackage[
backend=bibtex,
style=alphabetic,
sorting=ynt
]{biblatex}
\newtheorem{theorem}{Theorem}[section]
\newtheorem{thmx}{Theorem}

\newtheorem{lemma}[theorem]{Lemma}
\newtheorem{conjecture}[theorem]{Conjecture}
\theoremstyle{definition}
\newtheorem*{remark}{Remark}
\theoremstyle{definition}
\newtheorem*{claim}{Claim}
\theoremstyle{definition}
\newtheorem{example}{Example}
\theoremstyle{definition}

\theoremstyle{definition}

\theoremstyle{proposition}
\newtheorem{proposition}{Proposition}[section]
\theoremstyle{proposition}

\theoremstyle{definition}
\newtheorem*{notation}{Notation}

\title{Efficient Cycles in Loop Space}
\author{Robin Elliott}
\date{}

\begin{document}

\maketitle

\begin{abstract}
   This paper investigates how the geometry of a cycle in the loop space of Riemannian manifold controls its topology. For fixed $\beta \in H^n(\Omega X; \mathbb{R})$ one can ask how large $|\langle \beta, Z \rangle|$ can be for cycles $Z$ supported in loops of length $\leq L$ and of volume $\leq L^{n-1}$ for a suitably defined notion of volume of in loop space. We show that an upper bound to this question provides upper bounds Gromov's distortion of higher homotopy groups. We also show that we can exhibit better lower bounds than are currently known for the corresponding questions for Gromov's distortion. Specifically, we show there exists a $\beta$ detecting the homotopy class of the puncture in $[(\mathbb{CP}^2)^{\#4} \times S^2]^\circ$ and a family of cycles $Z_L$ with the geometric bounds above such that $|\langle \beta, Z \rangle| = \Omega(L^6/\text{log}L)$.
\end{abstract}

\section{Introduction} \label{intro}

The purpose of this paper is to investigate the existence of cycles in loop space with small volume, small suplength and large homological degree. If $Z$ is a chain in the loop space $\Omega X$ of a metric space $X$, we say that $\text{suplength}(Z) \leq L$ if $Z$ is supported in the space $\Omega_{\leq L}X$ of loops in $X$ of length at most $L$. The notion of volume of a chain will be defined using the $n$-dimensional Hausdorff measure of the chain, with respect to a metric on $\Omega X$ induced by the metric on $X$. The homological degree of a cycle will be measured by evaluating it against a fixed cohomology class $\beta \in H^n(\Omega X)$. We work with real coefficients throughout the paper unless otherwise stated.

More precisely, given $\beta \in H^n(\Omega X)$ define the \textit{cohomological distortion} of $\beta$
$$ \delta_{\beta}(L) = \text{sup} \{ \langle \beta,\; Z \rangle \;|\; \text{Suplength}(Z) \leq L \text{ and } \text{Vol}(Z) \leq L^n \} $$
and investigate the asymptotic growth of $\delta_{\beta}(L)$ as $L \to \infty$. Here the volume bound of $\leq L^n$ is chosen so that this notion of distortion matches existing notions of distortion, as described in Section \ref{gromovdistortion}.

\begin{example} Consider $\beta$ a generator of $H^{n-1}(\Omega S^n) \cong \mathbb{R}$ for $n\geq 2$. This $\beta$ detects the degree of a map $f: S^n \to S^n$, in the sense that if $\hat{f}: S^{n-1} \to \Omega S^n$ is a desuspension of $f$ (under the suspension-loop adjunction) then $\langle \beta,\; \hat{f}_*[S^{n-1}] \rangle = C\text{deg}(f)$ for some nonzero constant $C$ depending on the choice of generator $\beta$. We will show that $\delta_{\beta}(L) = \Omega(L^n)$.

Let $\zeta_n$ be a fixed cycle representing the generator of $H_{n-1}(\Omega S^n)$, i.e. $\zeta_n$ is a sweepout of $S^n$ by loops. For $L \in \mathbb{N}$, let $\{L\}: \Omega X \to \Omega X$ be the map that sends a loop $\gamma$ to its $L$-fold concatenation, $\gamma \cdot \gamma \cdot \ldots \cdot \gamma$. Once the relevant metric setup has been established (see section \ref{metricsetup}), it can be shown that $L^{n-1}\{L\}_*\zeta_n$ are a family of $(n-1)$-cycles in $\Omega S^n$ witnessing that $\delta_{\beta}(L) \gtrsim L^n$.

Here is another construction of a different family giving the same bound. Let $f_L: S^n \to S^n$ be a family of $L$-Lipschitz maps that are degree $\Theta(L^n)$. Then take the family of cycles to be $(\Omega f_L)_*\zeta_n$. This gives the same asymptotic lower bound of $L^n$ for $\delta_{\beta}(L)$. In \cite{itintinqt} it is shown that this bound of $L^n$ is asymptotically sharp. 
\end{example}

\begin{example} Consider $\beta$ a generator of $H^{2n-2}(\Omega S^n) \cong \mathbb{R}$ for $n \geq 2$ even. This $\beta$ detects the Hopf invariant of a map $S^{2n-1} \to S^n$ in the same sense as above. We will show that its cohomological distortion is $\Omega(L^{2n-1})$.

Consider the same $\zeta_n$ as in the previous example. The $(2n-2)$-cycle $[\zeta_n, \zeta_n]$ generates $H_{2n-2}(\Omega X)$, where here the bracket is the graded commutator associated to the Pontryagin product on $C_*(\Omega X)$. Then the family of cycles $Z_L \coloneqq L^{2n-2}[\{L\}_*\zeta_n, \{L\}_*\zeta_n]$ witnesses that $\delta_{\beta}(L) \gtrsim L^{2n}$.

Alternatively, we could instead take a family of maps $f_L: S^{2n-1} \to S^n$ that are $L$-Lipschitz and have Hopf invariant $\Theta(L^{2n})$. The same construction $Z_L \coloneqq (\Omega f_L)_*\zeta_{2n-2}$ as in the previous example also gives the cohomological distortion of $\beta$ is $\gtrsim L^{2n}$. In \cite{itintinqt} it is shown that this bound of $L^{2n}$ is asymptotically sharp.
\end{example}

\subsection{Relation to Gromov's distortion} \label{gromovdistortion}

The notion of cohomological distortion is related to Gromov's notion of homotopical distortion, which has its origins in \cite{htpicaleffects}. We reformulate Gromov's notion, following \cite{scalspaces}. Given a rational homotopy class $\alpha \in \pi_n(X) \otimes \mathbb{Q}$, the \textit{distortion} $\delta_{\alpha}(L)$ of $\alpha$ is
$$ \delta_{\alpha}(L) = \text{sup}
\{k \;|\; \text{there exists an $L$-Lipschitz map $S^n \to X$ with $[f] = k\alpha$} \} $$

The relationship between the two notions of distortion comes from the following observation. There is a constant $C>0$ such that any $L$-Lipschitz map $f: S^n \to X$ induces an $(n-1)$-cycle $\widehat{f}_*[S^{n-1}]$ in $\Omega X$ that is suplength at most $CL$ and volume at most $CL^n$. This gives the following proposition.

\begin{proposition} \label{distortioncomparison} Let $\beta \in H^{n-1}(\Omega X, \mathbb{R})$ and $\alpha \in \pi_n(X) \otimes \mathbb{Q}$. The dual of the Hurewicz homomorphism defines a real homotopy functional $\tau^{\vee}(\beta) \in \text{Hom}(\pi_n(X); \mathbb{R})$. If $\langle \tau^{\vee}(\beta), \alpha \rangle \neq 0$ then $\delta_{\beta}(L) \gtrsim \delta_{\alpha}(L)$.
\end{proposition}

Note that bound of $L^n$ on the $n$-volume of $Z$ in the definition of cohomological distortion was chosen so the inequality in Proposition \ref{distortioncomparison} does not have to contain any additional factors of $L$. 

In \cite{qhtypthy} Gromov conjectured that the distortion of a rational homotopy class $\alpha \in \pi_n(X) \otimes \mathbb{Q}$ was determined by the minimal model of $X$. In the case of the least distorted homotopy elements, a strong interpretation of this can be stated as follows. 

\begin{conjecture} (Gromov) The homotopical distortion of a rational homotopy class $\alpha \in \pi_n(X) \otimes \mathbb{Q}$ is $\Theta(L^n)$ if and only if the Hurewicz image of $\alpha$ is nonzero. Otherwise, the homotopical distortion is $\Omega(L^{n+1})$.

\end{conjecture}

Recently in \cite{scalspaces} Manin and Berdnikov disproved  the above conjecture. Specifically, they exhibited a punctured $6$-manifold $Y=[(\mathbb{CP}^2)^{\#4} \times S^2]^\circ$ and show the class of the puncture $\alpha \in \pi_5(Y) \otimes \mathbb{R}$ is in the kernel of the Hurewicz homomorphism but has homotopical distortion $o(L^6)$. The best current lower bound on the distortion of $\alpha$ is $\Omega(L^5)$, given by composing $\alpha$ with degree $L^5$ $L$-Lipschitz self maps of $S^n$. 

\subsection{Statement of results}

In this paper we investigate the homotopy class from the end of previous section under the lens of cohomological distortion. 

\begin{thmx}\label{thma} Let $Y = [(\mathbb{CP}^2)^{\#4} \times S^2]^\circ$ and $\alpha \in \pi_5(X) \otimes \mathbb{Q}$ be the class of the puncture. Then there exists a $\beta \in H^4(\Omega Y)$ such that its image under the dual Hurewicz detects $\alpha$ and such that $\beta$ has cohomological distortion $\Omega(L^6/\text{log}L)$.
\end{thmx}

Theorem \ref{thma} is proved by explicitly constructing a $\beta$ in terms of iterated integrals, and explicitly constructing an efficient family of $4$-cycles $Z_L$ on $\Omega Y$. At the heart of the construction of $Z_L$ is the construction of the following efficient (small volume) homology which exists in the loop space of any Riemannian manifold $X$, not just the $Y$ stated above.

\begin{thmx}\label{thmb} Let $X$ be a Riemannian manifold and let $L > 0$ be a power of two. For any $ n\geq 1 $, fix $n$-cycles $[S^n]$ representing the fundamental class of $S^n$. Let $f_1: S^{n_1+1} \to X$ and $f_2: S^{n_2+1} \to X$ be Lipschitz maps. For $i=1,2$ let $Z_i$ be the $n_i$-cycle $(\widehat{f_i})_*[S^{n_i}]$ in $\Omega X$, where $\widehat{f_i}$ is the desuspension of $f_i$. Let $[\cdot,\;\cdot]$ denote the graded commutator associated to the Pontryagin product on chains on  $\Omega X$.  Then the two homologous $(n_1+n_2)$-cycles
$$[\{L\}_*Z_1, \{L\}_*Z_1] \;\;\;\text{and}\;\;\; L\{L\}_*[Z_1, Z_2]$$ admit a homology $P = P(f_1, f_2, L)$ with suplength at most $CL$ and volume $CL\text{log}L$. Here $C > 0$ is a constant depending only on $n_i$, the Lipschitz constants of the maps $f_i$, and the choice of metric on $S^{n_i}$. In particular, $C$ is independent of $L$ and $X$. Moreover, if the image of $f_i$ lie in some skeleton $X^{(k)}$ of $X$ then $P$ can be taken to be supported in the subspace $\Omega X^{(k)} \subset \Omega X$.
\end{thmx}

The punchline of Theorem \ref{thmb} is that we can find a $\lesssim L$-suplength homology between these two cycles with volume at most a constant times $L\text{log}L$. A weaker bound of $L^2$ is easier to derive, but then only gives a lower bound of the cohomological distortion of $\beta \in H^4(\Omega Y)$ of $L^5$. This is the same bound that comes from the trivial bound on the homotopical distortion of $\alpha$. It is not currently known if Theorem \ref{thmb} could be improved to give a linear bound on the volume of $Y$. If so, it would give a lower bound on the cohomological distortion of $\beta$ of $L^6$, strictly asymptotically larger than Berdnikov and Manin's upper bound of $o(L^6)$ for the homotopical distortion of $\alpha$.

\subsection{Outline of the paper}

In Section 2, we set up the geometric and metric notions needed in the paper. Section 3 contains the proof of Theorem \ref{thmb}, i.e. the construction of a small-volume homology between a certain pair of homologous cycles in $\Omega X$ for any $X$. Section 4 contains the proof of Theorem \ref{thma}: the construction of the family of cycles in $\Omega Y$ for $Y=[(\mathbb{CP}^2)^{\#4} \times S^2]^\circ$ and the the geometric and topological bounds on these cycles.

\subsection{Acknowledgements}

I would like to thank my advisor, Larry Guth, for constant advice and enthusiasm. I would also like to thank Fedya Manin for helpful conversations and explanations; much of this work is inspired by the content and questions in \cite{shadows} and \cite{scalspaces}. I also benefited from Dev Sinha's seminar at MIT in Spring 2020 and the geometric insights exposited there. Finally I would like to thank Haynes Miller and Luis Kumandari for helpful conversations and useful comments. 

\section{Metric setup} \label{metricsetup}

Let $X$ be a Riemannian manifold with basepoint $x_0$. By $\Omega X$ we mean the space of piecewise smooth \textit{Moore loops}, i.e. a point in $\Omega X$ is a pair $(a, \gamma)$ with \textit{curfew} $a \geq 0$ a real number and piecewise smooth $\gamma: [0, a] \to X$ with $\gamma(0) = \gamma(a) = x_0 \in X$. In this model $\Omega X$ is a strictly associative $H$-space with $n$-ary multiplication
\[ \mu_n: (\Omega X)^{n} \to \Omega X \]
given by summing of curfews and concatenation of loops.

Let $C_*(\Omega X)$ denote the (real) cubical chain complex. The concatenation multiplication $\mu_2: \Omega X \times \Omega X \rightarrow \Omega X$ induces a multiplication \[\cdot: C_m(\Omega X) \otimes C_n(\Omega X) \rightarrow C_{m+n}(\Omega X)\] on chains. The graded commutator associated to this multiplication is given by
\[ [Z_1, Z_2] \coloneqq Z_1 \cdot Z_2 - (-1)^{|Z_1||Z_2|} Z_2 \cdot Z_1. \]

\begin{remark}
The notion of curfew is a technical requirement to ensure that the Jacobi relation is satisfied on the nose for the Lie bracket that will be defined on $C_*(\Omega X)$ below. We will always be considering cycles of constant curfew, although the size of this curfew may vary. The curfew is unimportant to the geometry, and is also distinct from the much more important quantity \textit{suplength} defined in Section \ref{intro}.
\end{remark}

The distance metric on $\Omega X$ we will use is given by
\[d_{\Omega X}((a_1, \gamma_1), (a_2, \gamma_2)) \coloneqq |a_1 - a_2| + \sup_{t \in [0,1]} d_X(\gamma_1(ta_1), \gamma_2(ta_2)). \]

This metric on $\Omega X$ gives an $n$-dimensional Hausdorff measure on $\Omega X$ for each $n$. For a simplex $\sigma$ in $\Omega X$, define the $n$-volume of $\sigma$ as the $n$-dimensional Hausdorff measure of its image. Given an arbitrary real chain $c = \sum \lambda_i \sigma_i$, define the volume of $c$ by $\sum |\lambda_i| \text{Vol}(\sigma_i)$.

There are other ways to define the volume of a chain. For example, the Riemannian metric $g$ on $X$ gives $\Omega X$ the structure of an infinite-dimensional Finsler manifold in the following way. Given a piecewise smooth loop $\gamma \in \Omega X$ with curfew $a$, the tangent space at $\gamma$, denoted $T_{\gamma}\Omega X$, is the space of piecewise smooth vector fields along $\gamma$ which vanish at the endpoints of $\gamma$. Given such a vector field $V$, define the norm on $V$ to be $$||V||_{\infty} = \sup_{t \in [0,a]}||V(t)||_{(X,g)} $$ where the norm on the right hand side is from the Riemannian metric. From here, if we assume our chain $Z$ is a pseudomanifold then we can pull back the Finsler metric to $Z$ and evaluate the volume of $Z$ using this pullback metric.

These two notions of volume agree, up to a constant. The important point that will be used in this paper is that given two chains $Z_1$ and $Z_2$ of dimensions $m$ and $n$ respectively, there exists a constant $C$ depending only on the dimensions of $Z_1$ and $Z_2$ such that $\text{Vol}_{m+n}(Z_1 \cdot Z_2) \leq C \text{Vol}_{m}(Z_1) \text{Vol}_{n}(Z_2)$. 

Of interest to us will be the map $\{L\}: \Omega X \to \Omega X$ given by (for any $L \in \mathbb{N}$) the composite of the $L$-fold diagonal and $L$-fold multiplication
\[ \Omega X \xrightarrow{\Delta_L} (\Omega X)^L \xrightarrow{\mu_L} \Omega X. \]

On the constant-curfew subspaces of $\Omega X$, $\{L\}$ is $1$-Lipschitz and so $\{L\}_*$ is non-increasing on the volume of constant-curfew chains. It does however multiply the suplength of chains by a factor of $L$, and multiply the curfew of points by a factor of $L$.

Given two chains $Z_1, Z_2$ with suplengths $L_1, L_2$ and volumes $V_1, V_2$ respectively, the chain $[Z_1, Z_2]$ has suplength at most $L_1+L_2$ and volume at most $C V_1 V_2$ for some constant $C$ depending only on the dimensions of the chains.

A \textit{spherical cycle} is one that is equal to $f_*[S^n]$ for some map $f: S^n \to \Omega X$ and some cycle $[S^n]$ representing the fundamental class of $S^n$. Such cycles are \textit{primitive} (they evaluate to zero on all nontrivial cup products), but the converse is not true: for example the graded commutator $[Z_1, Z_2]$ of two spherical is a cycle that is primitive but not spherical. The two notions do agree on the level of homology: a primitive cycles is homologous to a spherical cycle by the Milnor-Moore theorem \cite{milnormoore}.

Finally, recall Samelson's theorem \cite{samelson}. Let $\tau: \pi_{n+1}(X) \to H_n(\Omega X)$ be the Hurewicz map. Given $\alpha_1 \in \pi_{n_1+1}(X)$ and $\alpha_2 \in \pi_{n_2+1}(X)$ for $n_i \geq 1$, denote their Whitehead product by $[\alpha_1, \alpha_2]_{\text{Wh}} \in \pi_{n_1+n_2+1}(X)$. Then
$$ \tau [\alpha_1, \alpha_2]_{\text{Wh}} = (-1)^{n_1} [\tau \alpha_1,\; \tau \alpha_2] $$ 
where the bracket on the right hand side is the graded commutator in on $H_*(\Omega X)$.

\section{The proof of Theorem \ref{thmb}}

Here we restate Theorem \ref{thmb}.

\begin{theorem} Let $X$ be a Riemannian manifold and let $L > 0$ be a power of two. For any $ n\geq 1 $, fix $n$-cycles $[S^n]$ representing the fundamental class of $S^n$. Let $f_1: S^{n_1+1} \to X$ and $f_2: S^{n_2+1} \to X$ be Lipschitz maps. For $i=1,2$ let $Z_i$ be the $n_i$-cycle $(\widehat{f_i})_*[S^{n_i}]$ in $\Omega X$, where $\widehat{f_i}$ is the desuspension of $f_i$. Then the two homologous $(n_1+n_2)$-cycles
$$[\{L\}_*Z_1, \{L\}_*Z_1] \;\;\;\text{and}\;\;\; L\{L\}_*[Z_1, Z_2]$$ admit a homology $P = P(f_1, f_2, L)$ with suplength at most $CL$ and volume $CL\text{log}L$. Here $C > 0$ is a constant depending only on $n_i$, the Lipschitz constants of the maps $f_i$, and the choice of metric on $S^{n_i}$. In particular, $C$ is independent of $L$ and $X$. Moreover, if the image of $f_i$ lie in some skeleton $X^{(k)}$ of $X$ then $P$ can be taken to be supported in the subspace $\Omega X^{(k)} \subset \Omega X$.
\end{theorem}

\begin{proof} The $L_0$-Lipschitz condition on $f_1$ and $f_2$ implies that $Z_1$ and $Z_2$ have volumes $\text{Vol}(Z_1) \lesssim L_0^{n_1}$ and $\text{Vol}(Z_2) \lesssim L_0^{n_2}$ respectively. To prove the theorem, we will construct pieces of $P$ on many "scales" and the final $P$ will be the sum of the pieces. At a single scale we will construct the homology given in the following lemma.

\begin{lemma} There is a constant $C=C(n_1, n_2)>0$ such that the following holds. Let $Z_1, Z_2$ be constant-curfew spherical $n_1$- and $n_2$-cycles respectively with the same curfew, which are the images of $L_0$ Lipschitz maps $\widehat{f_i}: S^{n_i} \to \Omega X$, and such that $Z_1$ and $Z_2$ have suplength at most $L_1$. Then there is a constant-curfew $(n_1+n_2+1)$-chain $P' = P'(Z_1, Z_2)$ with
\[\partial P' = 2\{2\}_*[Z_1, Z_2] - [\{2\}_* Z_1, \{2\}_* Z_2] \]
and $P$ has suplength at most $4L_1$ and volume at most $CL_0^{n_1+n_2}(L_0+L_1)$.
\end{lemma}

Given the lemma, the homology $P$ satisfying the theorem is built out of the pieces $P'$ at each scale by:
\[ P \coloneqq \sum_{i=0}^{k-1} 2^i \{ 2^i \}_* P'(\{2^{k-1-i}\}_* Z_1, \{2^{k-1-i}\}_* Z_2). \]
In the $i$th summand, the chain $P'(\{2^{k-1-i}\}_* Z_1, \{2^{k-1-i}\}_* Z_2)$ has volume at most $C 2^{k-i-1} L_0^{n_1+n_2+1}$ and suplength at most $4 \cdot 2^{k-1-i} L_0$. So each of the $\text{log}L$ summands has suplength at most $4L L_0$ and volume at most $CLL_0^{n_1+n_2+1}$, as required.

It remains to prove the lemma.

\begin{proof} (of Lemma)
Let the common constant curfew of $Z_1$ and $Z_2$ be $a$. 

\begin{notation} From now on, we will drop the $\cdot$ in the multiplication on chains in $\Omega X$ and denote the multiplication by juxtaposition instead. We will denote the $0$-chain of a constant loop with curfew $a$ by $\bullet$. For example, $Z$ and $Z \bullet$ are similar chains, but to obtain $Z \bullet$ from $Z$ every point in the support of $Z$ is is post-concatenated with the constant loop of curfew $a$. In particular, $Z$ and $Z \bullet$ have the same volume and suplength.
\end{notation}

We will construct the homology in the lemma in three parts:
\begin{enumerate}
    \item First, a homology $P_1$ from  $[\{2\}_* Z_1, \{2\}_* Z_2]$ to an intermediate cycle $Q_1$.
    \item Next, a homology $P_2$ from $Q_1$ to another intermediate cycle $Q_2$. This homology will do nothing more than reparametrize loops in the support of $Q_1$. However, it is this homology that contributes most of the volume to $P$.
    \item Finally, a homology $P_3$ from $Q_2$ to $2\{2\}_*[Z_1, Z_2]$.
\end{enumerate}

To construct $P_1$, the crucial ingredient is the following. Recall that $Z_1 = (\widehat{f_1})_*[S^{n_1}]$, for $\widehat{f_1}: S^{n_1} \to \Omega X$. Moreover $\{2\}_*Z_1 = \mu_*(\widehat{f_1} \times \widehat{f_1})_*(\text{Diag}(S^{n_1}))$ where $\text{Diag}(S^{n_1})$ is the diagonal ${n_1}$-cycle $\Delta_*[S^{n_1}]$ in $S^{n_1} \times S^{n_1}$ and $\mu \coloneqq \mu_2: \Omega X \times \Omega X \to \Omega X$ is loop concatenation. The cycle $\text{Diag}(S^{n_1})$ is homologous to the cycle $\text{Bouquet}(S^{n_1}) = S^{n_1} \times \{*\} \cup \{*\} \times S^{n_1}$ in $S^{n_1} \times S^{n_1}$. Let $Y_{S^{n_1}}$ be a fixed homology between these. Then we can push this forward into $\Omega X$: we get the chain $\mu_*(\widehat{f_1} \times \widehat{f_1})_*(Y_{S^{n_1}})$ which is a homology between $\{2\}_*Z_1$ and $Z_1 \bullet + \bullet Z_1$.

Similarly, we can fix an analogous $(n_2+1)$-chain $Y_{S^{n_2}}$ and get a $(n_2+1)$-chain $\mu_* (\widehat{f_2} \times \widehat{f_2})_* (Y_{S^{n_2}})$ is a homology between $\{2\}_*Z_2$ and $Z_2 \bullet + \bullet Z_2$. The desired homology $P_1$ is the following $(n_1+n_2+1)$-chain: 
\[P_1 \coloneqq [\{2\}_* Z_1, \mu_*(\widehat{f_2} \times \widehat{f_2})_*(Y_{S^{n_2}})] + [\mu_*(\widehat{f_1} \times \widehat{f_1})_*(Y_{S^{n_1}}), Z_2 \bullet + \bullet Z_2]. \]

This is a homology between $[\{2\}_*Z_1, \{2\}_*Z_2]$ and $Q_1$, where $$Q_1 \coloneqq [Z_1 \bullet + \bullet Z_1,\; Z_2 \bullet + \bullet Z_2].$$

The suplength of $P_1$ is at most $4L_0$, as the suplength of each term in the bracket is at most $2L_0$. We also need to compute the volume of $P_1$. The maps $\widehat{f_1}$ and $\widehat{f_2}$ are $L_0$-Lipschitz, so similarly $\mu \circ (\widehat{f_1} \times \widehat{f_1})$ and $\mu \circ (\widehat{f_2} \times \widehat{f_2})$ are $L_0$-Lipschitz too. The chains $Y_{S^{n_1}}$ and $Y_{S^{n_1}}$ have fixed volumes $C(n_1)$ and $C(n_2)$ independent of $\widehat{f_1}$, $\widehat{f_1}$ and $X$. The cycle $\{2\}_*Z_1$ has volume equal to $\text{Vol}(Z_1)$ and the cycle $Z_2 \bullet + \bullet Z_2$ has volume equal to $2\text{Vol}(Z_2)$. Using $[Z', Z'']$ has volume at most $C \text{Vol}(Z')\text{Vol}(Z'')$, the volume of $P_1$ is at most $C L_0^{n_1+n_2+1}$.

Next we construct $P_3$, and will save $P_2$ to last. The idea is similar, except we need to use the analogue of $Y_{S^n}$ but for $S^{n_1} \times S^{n_2}$ and $S^{n_2} \times S^{n_1}$ in place of $S^{n_1}$ and $S^{n_2}$.  In the first part we found a homology from $\text{Diag}(S^{n_1})$ to a linear combination of cycles representing the standard basis of $H_{n_1}(S^{n_1} \times S^{n_1}) \cong H_{n_1}(S^{n_1}) \oplus H_{n_1}(S^{n_1})$. Here, we start with the diagonal $({n_1}+{n_2})$-cycle $\text{Diag}(S^{n_1} \times S^{n_2})$ in $(S^{n_1} \times S^{n_2}) \times (S^{n_1} \times S^{n_2})$. The group $H_{{n_1}+{n_2}}((S^{n_1} \times S^{n_2}) \times (S^{n_1} \times S^{n_2}))$ has a basis represented by
\begin{align*} S^{n_1} \times S^{n_2} \times \{*\} \times \{*\} ,\;\; \{*\} \times S^{n_2} \times S^{n_1} \times \{*\},\\
S^{n_1} \times \{*\} \times \{*\} \times S^{n_2},\;\; \{*\} \times \{*\} \times S^{n_1} \times S^{n_2}.
\end{align*}
Denote by $\text{Bouquet}(S^{n_1} \times S^{n_2})$ the sum of these that is homologous to $\text{Diag}(S^{n_1} \times S^{n_2})$, and let $Y_{S^{n_1} \times S^{n_2}}$ be a homology between them.

Now,
\begin{align*} \{2\}_*[Z_1, Z_2] = \; & (\mu_4)_*(\widehat{f_1} \times \widehat{f_2} \times \widehat{f_1} \times \widehat{f_2})_*\text{Diag}(S^{n_1} \times S^{n_2}) \\ 
&- (-1)^{n_1 n_2}(\mu_4)_*(\widehat{f_2} \times \widehat{f_1} \times \widehat{f_2} \times \widehat{f_1})_*\text{Diag}(S^{n_2} \times S^{n_1})
\end{align*}
and so the $(n_1+n_2+1)$-chain in $\Omega X$
\[P_3 \coloneqq (\mu_4)_*(\widehat{f_1} \times \widehat{f_2} \times \widehat{f_1} \times \widehat{f_2})_*Y_{S^{n_1} \times S^{n_2}} - (-1)^{n_1 n_2}(\mu_4)_*(\widehat{f_2} \times \widehat{f_1} \times \widehat{f_2} \times \widehat{f_1})_* Y_{S^{n_2} \times S^{n_1}} \]
gives the required chain for the third part. This a homology between $\{2\}_*[Z_1, Z_2]$ and $Q_2$, where
\begin{align*}
Q_2 \coloneqq \;& (Z_1 Z_2 \bullet \bullet + Z_1 \bullet \bullet Z_2 + \bullet Z_2 Z_1 \bullet +  \bullet \bullet Z_1 Z_2) \notag \\
& -(-1)^{mn} (Z_2 Z_1 \bullet \bullet + Z_2 \bullet \bullet Z_1 + \bullet Z_1 Z_2 \bullet +  \bullet \bullet Z_2 Z_1) 
\end{align*}
is the pushforward under the right products $\widehat{f_1}$ and $\widehat{f_2}$ of $\text{Bouquet}(S^{n_1} \times S^{n_2})$ and $\text{Bouquet}(S^{n_2} \times S^{n_1})$.

By the same argument as for $P_1$, the suplength of $P_3$ is at most $2L_0$ and the volume of $P_3$ is at most $CL_0^{m+n+1}$.

Finally, we construct the homology $P_2$. Note that both $Q_1$ (after expanding the definition of the brackt) and $Q_2$ are homologous to $4(Z_1 Z_2\bullet \bullet - (-1)^{n_1n_2}Z_2 Z_1 \bullet \bullet)$. Here the extra $\bullet$s have been inserted (arbitrarily) to make this chain have curfew $4a$, the same as $Q_1$ and $Q_2$. We will construct the homology for the second part one summand at a time: for each summand of $Q_1$ and $Q_2$ we will exhibit a homology to either $Z_1 Z_2\bullet \bullet$ or $Z_2 Z_1 \bullet \bullet$.

For example, let's construct a homology between $\bullet Z_1 Z_2 \bullet$ and $Z_1 Z_2 \bullet \bullet$. Let $\bullet_a (\Omega X)_{3a}$ denote the subspace of $\Omega X$ consisting of loops of curfew $4a$ which are the constant loop when restricted to $[0, a] \subset [0, 4a]$. Consider the linear interpolation homotopy $H_s: \bullet_a(\Omega X)_{3a} \times I \to \Omega X$ that linearly interpolates between $H_0$ the identity map, and the map $H_1$ that reparametrizes the loop $\gamma$ in the following way.
\[ H_1(\gamma: [0, 4a] \to X) =  \begin{cases} t \mapsto \gamma(t+a) &\mbox{if } t < 3a \\
t \mapsto x_0 & \mbox{if } t \geq 3a \end{cases} \]

That is, if $Z$ is a cycle with curfew $3a$ then $H_1$ send $\bullet Z$ to $Z \bullet$. By "linearly interpolates" we mean that for $s \in [0,1]$, $H_s$ does the reparametrization that linearly interpolates between the identity reparametrization and the reparametrization in $H_1$. Note that $H_s$ preserves curfew for each $s \in [0,1]$. The homotopy $H$ induces a chain homotopy $H_*$ such that $H_*(\bullet Z_1 Z_2 \bullet)$ is a homology between $\bullet Z_1 Z_2 \bullet$ and $Z_1 Z_2 \bullet \bullet$. Note that $H_*$ does not increase suplength and for any chain $Z$,
\[ \text{Vol}_{n+1}(H_*(Z)) \lesssim \text{Suplength}(Z) \text{Vol}_n(Z). \]

All other summands are similar.
\end{proof}

The lemma follows, and so the theorem does too.

\end{proof}

\section{The proof of Theorem \ref{thma}}

In this section we will build the efficient family of $4$-cycles $Z_L$ on $\Omega Y$ for $Y =[(\mathbb{CP}^2)^{\#4} \times S^2]^\circ$. For $L=1$, the cycle $Z_1$ will be homologous (perhaps up to rescaling) to $\tau(\alpha)$, the Hurewicz image of a sweepout by loops of the homotopy class of the puncture in $Y$. This will be done by showing that both $[Z_1]$ and $\tau(\alpha)$ lie in the $1$-dimensional subspace of primitive elements in $H_4(\Omega Y)$. For $\tau(\alpha)$ this is immediate from the Milnor-Moore theorem; for $Z_1$ it will follow from the definition of $Z_1$ outlined below.

The cycles $Z_L$ will be built out of some building-block chains on $\Omega Y$, the homologies constructed above, and the bracket $[\cdot,\;\cdot]$ on $C_*(\Omega Y)$. The building-block chains on $\Omega Y$, which we denote by $A_i$, $B$, $C_i$ and $D$ ($1 \leq i \leq 4)$ are the generators of the Adams-Hilton construction applied to the natural cell decomposition of $Y$. We explicitly construct them below; see \cite{adamshilton} for the construction in full generality.

The natural cell structure on $Y$ has five $2$-cells: one for each of the $2$-cells in $(\mathbb{CP}^2)^{\#4}$ and one from the $S^2$. Let's call these $2$-cells $\tilde{A}_i$ for $1\leq i \leq 4$ and $\tilde{B}$ respectively. Each of these $2$-cells has trivial attaching map, so they can be thought of as embedded $2$-spheres in $Y$. Take a constant-curfew sweepout of the $2$-sphere by loops, i.e. $1$-cycle $\zeta_1$ in $\Omega S^2$ generating $H_1(\Omega S^2)$ integrally. Then pushing forward $\zeta_1$ under the loop of each of the five embeddings $S^2 \hookrightarrow Y$ gives five $1$-cycles in $\Omega Y$ that we will denote by $A_i$ for $1\leq i \leq 4$ and $B$.

The cell structure on $Y$ also has five $4$-cells: one for the product of the $2$-cell in $S^2$ and each $2$-cell in $(\mathbb{CP}^2)^{\#4}$, and one corresponding to the top cell of $(\mathbb{CP}^2)^{\#4}$ (crossed with the $0$-cell in $S^2$). Let's call these $4$-cells $\tilde{C}_i$ for $1 \leq i \leq 4$ and $\tilde{D}$ respectively. We want to write down corresponding $3$-chains in $\Omega Y$, but this is slightly trickier to do as these $4$-cells have nontrivial attaching maps. The attaching map of $\tilde{C}_i$ is the Whitehead product $[\tilde{A}_i, \tilde{B}]_{\text{Wh}}$.\footnote{Here, by slight abuse of notation, we conflate the embedded $2$-sphere $\tilde{A}$ with a map $S^2 \to Y$ and similarly for the $\tilde{B}_i$.} The attaching map $\phi_D$ of $\tilde{D}$ is homotopic to $$\displaystyle \sum_{i=1}^4 [\tilde{A}_i, \tilde{A}_i]_{\text{Wh}}$$

Now we construct the $3$-chain $D$ in $\Omega Y$ corresponding to $\tilde{D}$. Take a constant-curfew $3$-chain $\xi_4$ in $\Omega D^4$ which corresponds to a sweepout of loops of the $4$-disc relative to its boundary; that is $\partial \xi_4 = \zeta_3$ is a sweepout of $\Omega S^3$. Pushing this forward under the loop map of the inclusion of the $4$-cell $\tilde{D}$ gives us a $3$-chain $D'$ in $\Omega Y$. The boundary of $D'$ in $\Omega Y$ is a $2$-cycle that is the sweepout of the attaching map $\phi_D$ of $\tilde{D}$. Now by Samelson's theorem \cite{samelson}, $\partial D$ is homologous to $\sum_{i=1}^4 [A_i, A_i]$, where now here the bracket denotes the commutator product in $C_*(\Omega Y)$ rather than the Whitehead bracket. Fix a homology $D''$ between them. Then our $4$-chain $D$ is the sum of $D'$ and $D''$. 

We proceed similarly to construct each $C_i$. Sweeping out each $4$-cell $\tilde{C}_i$ gives a $3$-chain $C'_i$ with $\partial C'_i$ a $2$-cycle in $\Omega Y$ that is a sweepout of $[A_i, B]_{\text{Wh}}$. Take a homology $C''_i$ in $\Omega Y$ from $[\tilde{A}_i, \tilde{B}]_{\text{Wh}}$ to $[A_i, B]$. Then define $C_i \coloneqq C'_i + C''_i$.

The details of doing this process in full generality for any cell complex $X$ were worked out by Adams and Hilton \cite{adamshilton}, who use the cell decomposition of $X$ to give a small differential graded algebra that is quasi-isomorphic to $C_*(\Omega X)$. However the full power of the Adams-Hilton construction is not needed here since the cell decomposition of our $Y$ was fairly simple: in particular $Y$ is homotopy equivalent to the cofiber of a map between wedges of spheres. 

\subsection{The construction of $Z_L$}
First we will define a cycle $Z_1$, which will then be modified to produce the cycles $Z_L$. Given the $1$-cycles $A_i$ and $B$ and $3$-chains $C_i$ and $D$ constructed in the previous section, $Z_1$ is the cycle
$$Z_1 \coloneqq [B,\;D] + 2\sum_{i=1}^4[A_i, C_i].$$
The boundary of this is $$\partial Z_1 = [B, \sum_{i=1}^4 [A_i, A_i]] + 2\sum_{i=1}^4[A_i, [B, A_i]]$$ which vanishes due to the Jacobi identity applied to $A_i$, $A_i$ and $B$.

Now we define $Z_L$. It will look similar to the definition of $Z_1$ but with each of the building-block chains $A_i$, $B$, $C_i$, $D$ replaced by a modification of each. Let $A_{i,L}$ denote $\{L\}_* A_i$, and similarly $B_L$ denote $\{L\}_*B$. Let $C_{i,L}$ be the sum of the chain $L\{L\}_*C_i$ and the homology $P(A_i, B, L)$ (constructed in the previous section from $\partial L\{L\}_*C_i$ to $[\{L\}_*A_i, \{L\}_*B]$). So $C_{i,L}$ has suplength $\lesssim L$, volume $\lesssim L\text{log}L$, and boundary $[\{L\}_*A_i, \{L\}_*B] = [A_{i,L}, B_L]$.

Similarly, for each $i$ take the homology $P(A_i, A_i, L)$ constructed in the previous section from $L\{L\}[A_i, A_i]$ to $[\{L\}_*A_i, \{L\}_*A_i]$. Then define the $3$-chain
$$ D_L \coloneqq L\{L\}_*D + \sum_{i=1}^4 P(A_i, A_i, L)$$
which has suplength $\lesssim L$, volume $\lesssim L\text{log}L$, and boundary $\sum_{i=1}^4[\{L\}_*A_i, \{L\}_*A_i] = \sum_{i=1}^4[A_{i,L}, A_{i,L}]$.

Then the $4$-cycle $Z_L$ is
$$Z_1 \coloneqq [B_L,\;D_L] + 2\sum_{i=1}^4[A_{i,L}, C_{i,L}].$$

The boundary of this is $$\partial Z_L = [B_L, \sum_{i=1}^4 [A_{i,L}, A_{i,L}]] + 2\sum_{i=1}^4[A_{i,L}, [B_L, A_{i,L}]]$$ which vanishes due to the Jacobi identity applied to $A_{i,L}$, $A_{i,L}$ and $B_L$.

To prove Theorem \ref{thma} we will show the following bounds on the suplength of $Z_L$, the volume of $Z_L$, and the homological degree of $Z_L$ (with respect to some $\beta \in H^4(\Omega Y)$).

\begin{proposition} \label{boundsprop} There exist constants $C, c > 0$ and a $\beta \in H^4(\Omega Y)$ such that the $4$-cycles $Z_L$ constructed above satisfy:
\begin{enumerate}
    \item $\text{Suplength}(Z_L) \leq CL$
    \item $\text{$Vol$}(Z_L) \leq CL\text{log}L$
    \item $\langle \beta, Z_L \rangle > cL^3$
\end{enumerate} \end{proposition}

The proof of this occupies the rest of the paper. We start with the suplength bound. By construction, the suplength of each chain $A_{i,L}$, $B_L$, $C_{i,L}$ and $D_L$ is $L$ times the suplength of the corresponding chain without the subscript $L$. Suplength is additive under the bracket, so the suplength bound clearly holds when $C$ is taken to be at least $$2\;\text{max} \{ \text{Suplength}(\Psi) \;|\; \Psi \in \{A_i, B, C_i, D\}\}.$$
Similarly we obtain the volume bound by inspecting the volume of each of $A_{i,L}$, $B_L$, $C_{i,L}$ and $D_L$. The cycles $A_{i,L} \coloneqq \{L\}_*A_i$ and $B_L \coloneqq \{L\}_B$ satisfy $\text{Vol}(A_{i,L}) \leq \text{Vol}(A_i)$ and $\text{Vol}(B_L) \leq \text{Vol}(B)$, as $\{L\}$ is $1$-Lipschitz on constant curfew subspaces of $\Omega X$ so nonincreasing in the volume of constant curfew chains. The chain $D_L$ is the sum of $L\{L\}_*D$ (volume $\leq L\;\text{Vol}(D)$ and a homology $Y$ constructed in the previous section (volume $\leq CL\text{log}L$). Hence $\text{Vol}(D_L) \leq CL\text{log}L$ for some $C$ and similarly with the $C_{i,L}$. The volume of the bracket is the product of the volumes (up to a constant factor) and so the volume bound on $Z_L$ follows.

The proof of the final part of the proposition occupies the next subsection.

\subsection{Homological degree of $Z_L$}

We will evaluate $Z_L$ against a cohomology class in $H^4(\Omega Y)$ using Chen's iterated integrals. These are differential forms defined on $\Omega Y$, built out of differential forms on the underlying space $Y$. For a rigorous setup of the smooth structure on $\Omega Y$ and discussion of the construction of iterated integrals, a good review is \cite{itintreview}. See also original papers \cite{itintloopspacehomology}, \cite{itpathints} a concise but thorough overview in Hain's thesis \cite{hainthesis}, and also \cite{itintinqt} for a further discussion of the interplay between iterated integrals on $\Omega X$ and a Riemannian metric on $X$.

The third part of Proposition \ref{boundsprop} is established via the following claim.

\begin{claim} There is a nonzero cohomology class $\beta \in H^4(\Omega Y)$ with the two following properties. First, the image of $\beta$ under the dual of the real Hurewicz map $H^4(\Omega Y) \rightarrow \text{Hom}(\pi_5(Y) \otimes \mathbb{R}, \mathbb{R})$ is nontrivial. Second, $\langle \beta, Z_L \rangle = L^3\langle \beta, Z_1 \rangle \neq 0$.
\end{claim}

Here's the $\beta$ that we will use. The $2$-cells $\tilde{A}_i$ and $\tilde{B}$ in $Y$ give a basis for $H_2(Y)$. Take differential forms $a_i, b$ on $Y$ that represent the dual basis of $H^2(Y)$, such that when restricted to the $2$-skeleton $Y^{(2)}$, each $a_i$ and $b$ is only supported on its respective $2$-cell and their supports do not contain the basepoint. Similarly, the  $4$-cells of $Y$ are $\tilde{C}_i$ and $\tilde{D}$. These give a basis of $H_4(Y)$. A dual basis of $H^4(Y)$ is represented by $c_i \coloneqq a_ib$ and $d \coloneqq a_1^2$ (the $a_i^2$ are all cohomologous). Note that $a_ib$ vanishes identically on the $2$-skeleton $Y^{(2)}$ of $Y$. Then consider the iterated integral $\smallint a_1 c_1$, a closed differential $4$-form on $\Omega Y$. Then $\beta$ is the cohomology class of this $\smallint a_1 c_1$.

Intuitively, this cohomology class can be thought of as follows. Fix a relative $4$-submanifold $PD(a_1)$ in $(Y, \partial Y)$ Poincar\'e dual to $a_1$ $\tilde{A}_1$, and relative $2$-submanifold $PD(c_1)$ in $(Y, \partial Y)$ Poincar\'e dual to $c_1$. Given a $4$-cycle $Z$ in $\Omega Y$. the pairing $\langle \smallint c_1 a_1, Z \rangle$ counts, with appropriate multiplicity, the number of loops in the support of $Z$ that first pass through $PD(c_1)$ (a codimension $3$ condition) and then pass through $PD(a_1)$ (a codimension $1$ condition). I would like to thank Dev Sinha for explaining this way of thinking to me.

First we show that under the dual of the real Hurewicz, $\beta$ is nontrivial in $\text{Hom}(\pi_5(Y)\otimes \mathbb{R}, \mathbb{R})$. We do this by showing that the cycle $Z_1$ is primitive (so its homology class in the image of the real Hurewicz map $\pi_4(\Omega Y) \rightarrow H_4(\Omega Y)$) and showing that $\langle \beta, Z_1 \rangle \neq 0$.

The cycle $Z_1$ is primitive for algebro-topological reasons. The chains $A_i, B, C_i, D$ generate the Adams-Hilton chain algebra $AH(Y)$ which is a model for $C_*(\Omega Y)$. Working rationally, the primitively generated Hopf algebra structure on $AH(Y)$ gives an isomorphism between $H_*(AH(Y))$ and $H_*(\Omega Y)$ as Hopf algebras \cite{ratlhtpypthy} \cite{milnormoore}. The upshot of this is that any cycle constructed out of $A_i$, $B$, $C_i$, $D$ and the Lie bracket on cycles is primitive.

Next we will evaluate $\beta$ our cycles $Z_1$ and $Z_L$. To do this we will employ the following two lemmas about evaluation of iterated integrals.

\begin{lemma} (Chen, \cite{itintloopspacehomology}) Let $M$ be a simply connected smooth manifold with finite Betti numbers. Let $Z_1$ and $Z_2$ be chains on $\Omega M$ and $Z_1 Z_2$ their product. If $\omega_1, \omega_2, \dots,  \omega_r$ are smooth forms on $M$ then
$$ \langle \smallint \omega_1 \omega_2 \dots \omega_r,\; Z_1 Z_2 \rangle = \sum_{i=0}^r \langle \smallint \omega_1 \dots \omega_i,\; Z_1 \rangle \langle \smallint \omega_{i+1} \dots \omega_r,\; Z_2 \rangle.$$
\end{lemma}

In both lemmas, the convention is that the paring vanishes if the dimension of the differential form is not equal to the dimension of the chain that it is being evaluated.

\begin{lemma} \label{suspensionlemma} (Chen, \cite{itpathints}) Let $M$ be a simply connected smooth manifold with finite Betti numbers. Let $Z$ be a $n$-chain on $\Omega M$ with suspension $\hat{Z} \in C_{n+1}(M)$. Let $\omega$ be a differential form on $M$. Then 
$\langle \smallint w,\; Z \rangle = \langle \omega,\; \hat{Z} \rangle$.
\end{lemma}

Using these lemmas and the description of $Z_1 = [B, D] + 2\sum_{i=1}^4[A_i, C_i]$ we can compute:
\begin{align*}
    \langle \smallint a_1 c_1,\; Z_1 \rangle
    &=  \langle \smallint a_1 c_1,\; BD + DB + 2\sum_{i=1}^4(A_iC_i + C_iA_i)\rangle \\
    &=  \langle \smallint a_1,\; B \rangle \langle \smallint c_1,\; D \rangle + 2\sum_{i=1}^4 \langle \smallint a_1,\; A_i \rangle \langle \smallint c_1,\; C_i \rangle \\
    &= 0 \cdot 0 + 2 \sum_{i=1}^4 \delta_{i1} \cdot \delta_{i1} \\
    &= 2 \neq 0.
\end{align*}

Next, we evaluate $\smallint a_1 c_1$ on $Z_L$. Recall $Z_L = [B_L, D_L] + 2 \sum_{i=1}^4[A_{i,L}, C_{i,L}]$. 
\begin{lemma}
Let $X$ be an arbitrary topological space and let $Z$ be a spherical $n$-cycle in $\Omega X$. Then on the level of homology, $[\{L\}_*Z] = L[Z]$.
\end{lemma}

\begin{proof} Let $f: S^n \to \Omega X$ be a map and $\omega_n$ an $n$-cycle representing the fundamental class such that $f_*\omega_n = Z$. Let $Y_{n,L}$ be an $(n+1)$-chain in $(S^n)^{\times L}$ with boundary $\text{Diag}(\omega_n) - \text{Bouquet}(\omega_n)$. Then $(f \times f \times \dots \times f)_*Y_{n,L}$ is a homology between $\{L\}_*Z$ and a cycle that, after a reparametrization of loops, is homologous to $LZ$.
\end{proof}

Hence on the level of homology $[B_L] = L[B]$ and similarly $[A_{i,L}] = L[A_i]$.  Since the $1$-form $\smallint a_i$ is closed on $ \Omega X$, this gives us that $\langle \smallint a_i,\; B_L \rangle = 0$ and $\langle \smallint a_i,\; A_{i,L} \rangle = L\cdot\delta_{1i}$. This gives
\begin{align} \label{evaluateZ_L}
\langle \smallint a_1 c_1, \; Z_L \rangle &=  0 \cdot \langle \smallint c_1, D \rangle + 2 L \cdot \langle \smallint c_1,\; C_{1,L} \rangle
\end{align} 

It remains to compute $\langle \smallint c_1,\; C_{1,L} \rangle$. The $3$-chain $C_{1,L}$ is the sum of $L\{L\}_*C_1$ and a $3$-chain $Y$ constructed in the proof of Theorem \ref{thmb}. In particular, $P$ is supported in the subspace $\Omega (Y^{(2)}) \subset \Omega Y$ because the $A_i$ and $B$ are supported in $\Omega (Y^{(2)})$ too. The differential $4$-form $c_1$ on $Y$ vanishes on $Y^{(2)}$, and hence $\smallint c_1$ vanishes on $\Omega (Y^{(2)})$. Thus
$$ \langle \smallint c_1,\; C_{1,L} \rangle = \langle \smallint c_1,\; L\{L\}_*C_1 \rangle = L\langle \smallint c_1,\; \{L\}_*C_1 \rangle.$$

We can evaluate $\langle \smallint c_1,\; \{L\}_*C_1 \rangle$ using Lemma \ref{suspensionlemma}. The suspension $\widehat{C_1}$ is the $4$-cell $\tilde{C}_1$ and the suspension $\widehat{\{L\}_*C_1}$ is $L$ copies of $\tilde{C}_1$ stuck together. For a $4$-chain $Z$ in $Y$ with boundary lying in $Y^{(2)}$, $\langle c_1, Z \rangle$ measures the degree (relative to $Y^{(2)}$ over the $4$-cell $\tilde{C}_i$. Hence $\langle c_1,\; \widehat{\{L\}_* C_1} \rangle = L$. Plugging this back into (\ref{evaluateZ_L}) gives $\langle \smallint a_1 c_1,\; Z_L \rangle = 2L^3 = L^3 \langle \smallint a_1 c_1,\; Z_1 \rangle$ as required. 

\printbibliography

\medskip

\textsc{Department of Mathematics, Massachusetts Institute of Technology, Cambridge, MA, United States} \\
E-mail address: \texttt{relliott@mit.edu}

\end{document}